\documentclass[a4paper]{amsart}

\usepackage{amssymb}
\usepackage{amsthm}
\usepackage{amsmath, a4wide}
\usepackage{amsbsy}
\usepackage[all]{xy}
\usepackage{bm}
\usepackage{hyperref, color}
\usepackage{graphics}
\usepackage{graphicx}
\usepackage{tikz}
\usepackage{tkz-fct} 
\DeclareMathOperator{\wind}{wind}
\DeclareMathOperator{\supp}{supp}

\begin{document}

\newtheorem{theorem}{Theorem}
\newtheorem*{thm}{Theorem}
\newtheorem*{proposition}{Proposition}
\newtheorem{conjecture}{Conjecture}
\newtheorem{corollary}{Corollary}
\newtheorem{lemma}{Lemma}

\title[Nonlinear Phase Unwinding of  Functions]{Nonlinear Phase Unwinding of  Functions
}
\author{Ronald R. Coifman}
\address[Ronald R. Coifman]{Department of Mathematics, Program in Applied Mathematics, Yale University, New Haven, CT 06510, USA}
\email{coifman@math.yale.edu}

\author{Stefan Steinerberger}
\address[Stefan Steinerberger]{Department of Mathematics, Yale University, New Haven, CT 06510, USA}
\email{stefan.steinerberger@yale.edu}

\begin{abstract} We study a natural nonlinear analogue of Fourier series. Iterative Blaschke factorization allows one to formally write any
holomorphic function $F$ as a series which successively unravels or unwinds the oscillation of the function
$$ F = a_1 B_1 + a_2 B_1 B_2 + a_3 B_1 B_2 B_3 + \dots$$
where $a_i \in \mathbb{C}$ and $B_i$ is a Blaschke product. Numerical experiments point towards rapid
convergence of the formal series but the actual mechanism by which this is happening has yet to be explained. We derive a family of inequalities and use them to
 prove convergence for a large number of function spaces: for example, we have convergence in $L^2$ for functions in the Dirichlet space $\mathcal{D}$. Furthermore, we present a numerically efficient way to expand a function without explicit calculations of the Blaschke zeroes going back to Guido and Mary Weiss.
\end{abstract}

\maketitle

\section{Introduction}
\subsection{Blaschke factorization.} This paper studies a natural nonlinear way for unraveling the oscillation of a function $F:\mathbb{C} \rightarrow \mathbb{C}$ that is holomorphic in a neighborhood of the unit disk. Our starting point is a fundamental theorem in complex analysis (\textit{Blaschke factorization}) stating that any such function can be
decomposed as
$$ F = B \cdot G,$$
where $B$ is a Blaschke product, that is a function of the form
$$ B(z) = z^m\prod_{i \in I}{\frac{\overline{a_i}}{|a_i|}\frac{z-a_i}{1-\overline{a_i}z}},$$
where $m \in \mathbb{N}_{0}$ and $a_1, a_2, \dots \in \mathbb{D}$ are zeroes inside the unit disk $\mathbb{D}$ 
and $G$ has no roots in $\mathbb{D}$. For $|z|=1$ we have $|B(z)| = 1$
which motivates the analogy 
$$B \sim \mbox{frequency and}~G \sim \mbox{amplitude}$$
for the function restricted to the boundary. However, the function $G$ need not be constant: it can be any function that never vanishes
inside the unit disk. If $F$ has roots inside the unit disk, then the Blaschke factorization $F = B \cdot G$ is going to be nontrivial (meaning
$B \not\equiv 1$ and $G \not\equiv F$). $G$ should be 'simpler' than $F$ because the winding number around the origin decreases and we will quantify this in many different
ways. 

\subsection{A formal series.} There is a natural way of iterating Blaschke factorization that is inspired by the power series expansion
of a holomorphic function in 0. Since $G$ has no zeroes inside $\mathbb{D}$, its Blaschke factorization is the trivial one $G = 1 \cdot G$, however,
the function $G(z)-G(0)$ certainly has at least one root inside the unit disk $\mathbb{D}$ and will therefore yield some nontrivial Blaschke factorization
$G(z) - G(0) = B_1 G_1$. Altogether, this allows us to write
\begin{align*}
 F &= B \cdot G\\
&= B \cdot (G(0) + (G(z) - G(0))) \\
&= B \cdot (G(0) + B_1 G_1) \\
&= G(0) B + B B_1 G_1.
\end{align*}
At least formally, an iterative application gives rise to what we call the \textit{unwinding series}
$$ F = a_1 B_1 + a_2 B_1 B_2 + a_3 B_1 B_2 B_3 + a_4 B_1 B_2 B_3 B_4+ \dots$$
This formal expansion first appeared in the PhD thesis of Michel Nahon \cite{nahon}. Given a general function $F$ it is not numerically feasible to actually compute the roots of the function; a crucial insight in \cite{nahon} is that this is not necessary -- one can numerically obtain the Blaschke product
in a stable way by using a method that was first mentioned in a paper of Guido and Mary Weiss \cite{ww} (see also \cite{cw}) and has been investigated with respect to stability by Nahon \cite{nahon} and Letelier and Saito \cite{let}.
Numerical investigation \cite{nahon} indicates that the formal series 
$$ F = a_1 B_1 + a_2 B_1 B_2 + a_3 B_1 B_2 B_3 + a_4 B_1 B_2 B_3 B_4+ \dots$$
will converge to the actual function and, generically, this seems to happen at an exponential rate. 

\subsection{An example}  The following example/picture is taken from the PhD thesis of Michel Nahon \cite{nahon}.
Let us consider the Blaschke factorization of a function given by the projection of a modulated Gaussian on
the boundary onto holomorphic functions
$$ F(e^{i\theta}) = P_{+}\left(e^{-(\theta-\pi)^2} \cdot e^{10 i \theta}\right).$$
Fig. 1 shows the curves $t \rightarrow F(e^{i \theta})$, $t \rightarrow B(e^{i \theta})$ and $t \rightarrow G(e^{i \theta})$ in the complex plane.
A lot of the oscillation (and almost the entire phase) is
transported from $F$ to $B$ leaving $G$ significantly simpler than $F$. It also serves as a good example of the heuristic
$$B \sim \mbox{frequency and}~G \sim \mbox{amplitude}.$$
Figure 1 shows the real and imaginary part of the original signal, its shape when interpreted as a curve $F:\mathbb{T} \rightarrow \mathbb{C}$ and the same information
for $B$ and $G$: $B$ captures most of the oscillation.
\begin{center}
\begin{figure}[h!]
\includegraphics[width = 0.9\textwidth]{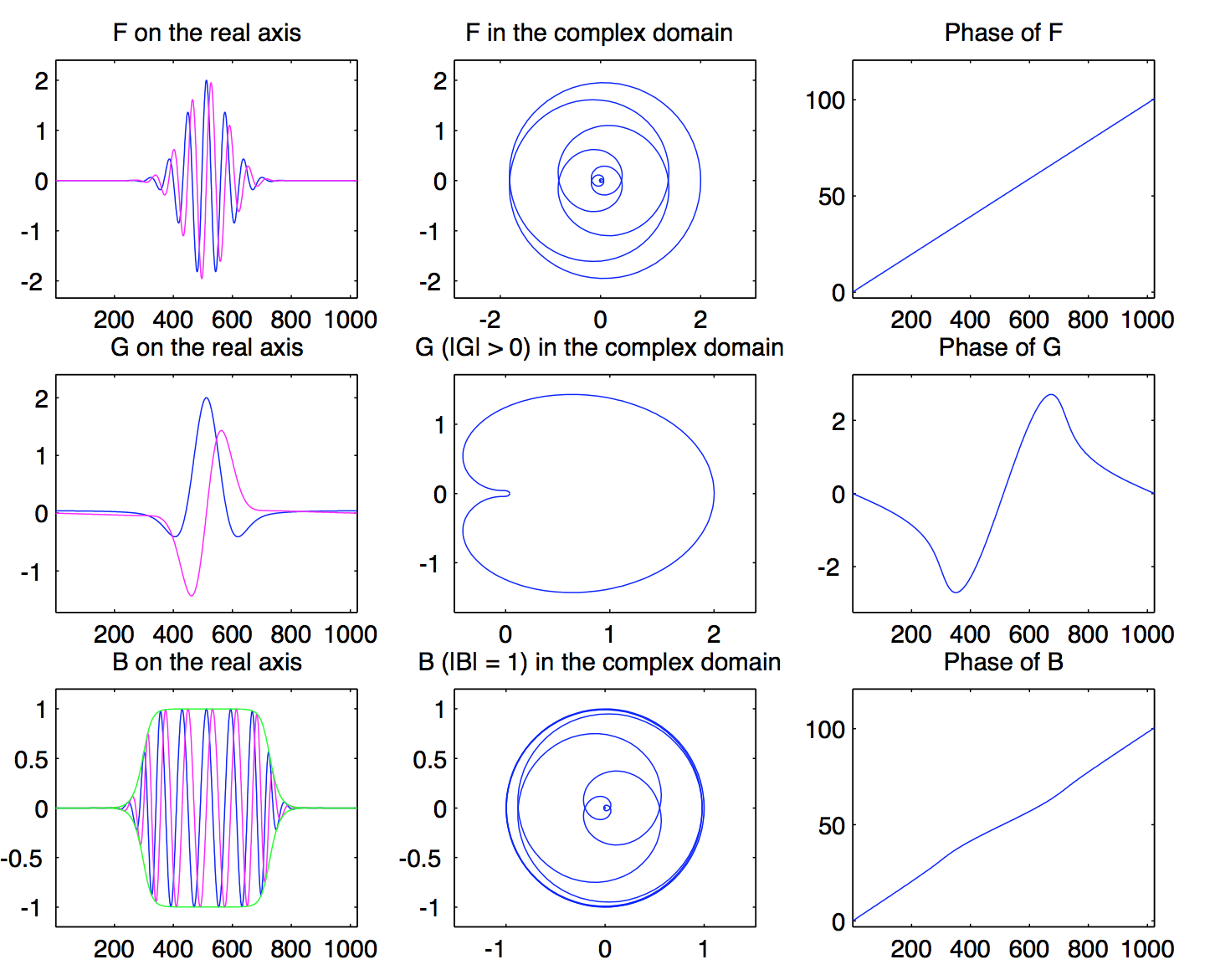}
\caption{A picture taken from Michel Nahon's thesis \cite{nahon}: the behavior of the Blaschke decomposition of $F$ on $\partial \mathbb{D}$.}
\end{figure}
\end{center}

\subsection{Related work.} Blaschke products have long been used in the signal analysis -- often under the name \textit{Malmquist-Takenaka} system. The crucial underlying
fact is that for any two Blaschke products $B_1, B_2$ the two functions $B_1$ and $B_1B_2$ are orthogonal on $L^2(\partial \mathbb{D})$ since
$$ \left\langle B_1, B_1 B_2 \right\rangle_{L^2(\partial \mathbb{D})} = \int_{0}^{2\pi}{\overline{B_1(e^{it})} B_1(e^{it}) B_2(e^{it}) dt} = \int_{0}^{2\pi}{B_2(e^{it}) dt} = 0.$$
This allows naturally to build orthogonal functions via $B_1, B_1B_2, B_1B_2B_3, \dots$ and contains the classical Fourier system $1, z, z^2, \dots$ as a special case.
We refer to papers of Eisner and Pap \cite{pape},  Feichtinger and Pap \cite{papf}, Pap \cite{pap} and Picinbino \cite{pic} for some examples. The unwinding series is first
studied in the PhD thesis of Michel Nahon \cite{nahon}. Subsequently, a method for numerical stabilization in the case of $|F(e^{i\theta})|$ becoming small has
been investigated by Letelier and Saito \cite{let}. The unwinding series has been used by Healy \cite{healy, healy2} in the study of the Doppler effect. Of particular 
importance is a paper of Tao Qian \cite{qtao5} in which he proves the convergence of the unwinding series for $F \in \mathcal{H}^2(\mathbb{T})$. This paper was brought to our attention
after this paper had been completed and we summarize his argument below. 
Closely related is also another approach developed by Qian, Ho, Leong and Wang \cite{qw} (and elaborated in further papers by Qian and collaborators \cite{qtao, qtao1, qtao2, qtao3}),
which they call \textit{adaptive Fourier transform}. The main idea is to use Blaschke products as a library and proceed by a projection pursuit approach, where at each
step one projects onto the element in the library yielding the largest inner product with the function:
$$ f_{n+1} = f_n - \left\langle f_n, B_n\right\rangle B_n $$
where $B_n$ is chosen among all Blaschke products with $n$ zeroes as the one yielding the largest inner product. Since, in particular, the functions $z^n$ are elements of that library, this approach
may be understood as a generalization of Fourier series -- among their results is also an independent rediscovery of the Guido and Mary Weiss algorithm \cite{qtao1} and of the unwinding series \cite{qtao4}.
There are also similarities in spirit with recent work of Mallat \cite{mallat}. Mallat's \textit{scattering transform} is a translation-invariant operator, which is Lipschitz-continuous w.r.t. to $C^1-$ 
diffeomorphisms of the underlying space. The construction is based
on an iterative application of wavelet transforms followed by restriction to the modulus. Our iterative application
$$ G_n(z) = G_n(0) + (G_n(z) - G_n(0)) = G_n(0) + B_{n+1}(z) G_{n+1}(z)$$
uses the modulus of the corresponding functions while the coefficients are given as the mean. This yields a comparable level of stability: at least the leading coefficient is stable under both 
perturbations of the function and reparametrization of the torus. 

\subsection{Notation and Outline.}  This paper deals with holomorphic 'signals' given as functions $f: \mathbb{T} \rightarrow \mathbb{C}$ by regarding
them as the restriction of a holomorphic function $F:\mathbb{C} \rightarrow \mathbb{C}$ on the boundary of the unit disk $\partial \mathbb{D}$. We will therefore use both $\mathbb{T}$
and $\partial \mathbb{D}$ depending on which aspect should be emphasized. We will work with both Sobolev spaces $H$ and Hardy spaces $\mathcal{H}$ (the Hilbert transform, which
appears only briefly, will be $\mathcal{H}_i$). $\mathcal{D}$ denotes the Dirichlet space, $P_+$ the holomorphic projection. \S 2 states the results, \S 3 gives background material
and discusses some possible applications. The proofs are given in \S 4.

\section{Statement of results}
\subsection{Setup.} Given a function $F:\mathbb{C} \rightarrow \mathbb{C}$, we define $G_1$ as the outer
part in the Blaschke factorization of $F$
$$ F = B_1 \cdot G_1$$
and then, iteratively, $G_{n+1}$ as the outer part in the Blaschke factorization
$$ G_n(z) - G_n(0) = B_{n+1}(z) G_{n+1}(z).$$
We are interested in ensuring that $\|G_{n}\|_{X} \rightarrow 0$ in some suitable space $X$ and our main statements
will be formulated that way. We emphasize that the formal series is, from the point of view of complex analysis, the canonical nonlinear extension of the Fourier series which arises from an iterative application of
$G_{n}(z)-G_{n}(0) = z\cdot G_{n+1}(z).$
The Blaschke series, in contrast to the Fourier series, proceeds by factoring out all zeroes inside $\mathbb{D}$ -- this
gives a rise to a much larger library of functions and makes it seem intuitive that one should not only expect convergence
but also faster convergence than for the Fourier series. At the same time, the iteration 
$$ G_n(z) - G_n(0) = B_{n+1}(z) G_{n+1}(z)$$
seems to define a very natural dynamical system on holomorphic function that could be of interest in its own right. 	

\subsection{Algorithm and roots.} We start (assuming for simplicity that there are no roots on the boundary of the unit circle) 
with two basic observations. Recall that a general Blaschke factor has the form
$$ B(z) = z^m\prod_{}{\frac{\overline{a_i}}{|a_i|}\frac{z-a_i}{1-\overline{a_i}z}}.$$
Suppose $F: \mathbb{C} \rightarrow \mathbb{C}$ is holomorphic, $F(0) \neq 0$ and $F$ has the set of roots $R = \left\{r_1, r_2, \dots \right\}$.  
If $ F= B \cdot G$, then the roots of $G$ are simply
$$ \begin{cases} r_i \qquad &\mbox{whenever}~|r_i| > 1 \\
\overline{1/r_i}\qquad &\mbox{whenever}~0<|r_i| < 1. \\
\end{cases}$$
Geometrically, this means that the roots of $F$ outside the unit circle stay unchanged while roots inside the unit circle are inverted across the unit circle. We emphasize that in every step of the algorithm consists of studying not $G_n(z)$ but $G_n(z) - G_n(0)$,
which will have a very different set of roots. However, one immediate easy consequence is the following.

\begin{proposition} Let $F: \mathbb{C} \rightarrow \mathbb{C}$ be given by a polynomial of degree $n$. Then the formal series
converges and is exact after $n$ steps.
\end{proposition}
\begin{proof}
The algorithm is closed in the set of polynomials. Furthermore, since $G_k(z) - G_k(0)$ has at least one root in $0$, the degree
of the polynomial decreases by at least 1 in every step.
\end{proof}
This argument is more algebraic than analytic and comes with the obvious limitation that it does not give any convergence speed
(analogously, it is not surprising that a trigonometric polynomial can be written as finite Fourier series).  An illustrative example is given by
$$ F(z) = (z-(1-\varepsilon))^{2n}.$$
The Blaschke factorization $F =  B \cdot G$ is easy to write down and 
$$ G(z) = (1-(1-\varepsilon)z)^{2n} = (1-\varepsilon)^{2n} \left(z- \frac{1}{(1-\varepsilon)}  \right)^{2n}.$$
By making $\varepsilon$ sufficiently small, the functions $F$ and $G$ can be made as close to each other in any reasonable function space as we wish.
These sort of examples immediately imply that it is not possible to construct a reasonable norm $X$ with 
$\|G_{n+1}\|_{X} \leq (1-\delta)\|G_{n}\|_{X}$ for some universal $\delta > 0$. 
Exponential convergence, which is observed in practice, will therefore either not always be the case or be the
consequence of an underlying phenomenon ensuring that iterative Blaschke factorization cannot always stay
close to set of functions behaving like these polynomials.

\subsection{Regularity assumptions.}  Blaschke factorization only guarantees a splitting into an \textit{inner} and an \textit{outer} function, where
the inner function itself is given by multiplying a Blaschke product and a singular inner function. It is not clear at this point how one would work
with a singular inner function and we are restricting the further scope of the paper to functions that are holomorphic on a domain $D \supset \mathbb{D}$
that contains an entire disk with radius $1+\varepsilon$, where $\varepsilon > 0$ can be arbitrarily small. This implies that the Blaschke factorization
really factors into a Blaschke product and an outer product; moreover, any (nonzero) function that is holomorphic in a neighborhood of the unit disk
has at most finitely many roots inside the unit disk, which guarantees that all Blaschke products are finite. This is not a serious restriction for applications
as most signals of interest can be approximated by a trigonometric polynomial -- it would be desirable to have a more complete theory from a 
mathematical perspective, however, at this point even the dynamics of iterative Blaschke factorization on polynomials, though convergent, is far from
being understood.

\subsection{A general contraction property.} This section presents our main convergence result. We first state the result in the most general form and comment
on special cases of particular interest further below. We start by introducing two norms on the Hardy space on the unit circle $\mathcal{H}^2(\mathbb{D})$.
Let $0 = \gamma_0 \leq \gamma_1 \leq \dots$ be an arbitrary monotonically increasing sequence of real numbers and let $X$
be the subspace of $\mathcal{H}^2( \mathbb{T})$ for which
$$ \left\|  \sum_{n \geq 0}{a_n z^n \big|_{\partial \mathbb{D}}} \right\|^2_{X} =  \left\|  \sum_{n \geq 0}{a_n e^{i n t}} \right\|^2_{X}  :=  \sum_{n \geq 0}{\gamma_n |a_n|^2} < \infty.$$
We define a second norm $Y$ (semi-norm whenever $\gamma$ is not strictly increasing)
$$ \left\|  \sum_{n \geq 0}{a_n z^n \big|_{\partial \mathbb{D}}} \right\|^2_{Y} =  \left\|  \sum_{n \geq 0}{a_n e^{in t}} \right\|^2_{Y} :=  \sum_{n \geq 0}{(\gamma_{n+1} - \gamma_n) |a_n|^2}.$$

Our main statement is that the Blaschke factorization acts nicely on these spaces. The first part of our statement is known (being ascribed to Digital Signal Processing in \cite{qtao1})
and can be equivalently phrased as follows: given a Blasche decomposition $F = B\cdot G$ and assuming both functions are expanded into a Fourier series
$$ F(z) = \sum_{n=0}^{\infty}{f_n z^n} \qquad \mbox{and} \qquad  G(z) = \sum_{n=0}^{\infty}{g_n z^n},$$
then, for every $N \in \mathbb{N}$
$$   \sum_{n \geq N}^{\infty}{|g_n|^2} \leq  \sum_{n \geq N}^{\infty}{|f_n|^2}.$$
Phrased differently, inner outer factorization shifts the energy to lower frequencies in a strictly monotonous way. Our main tool will
be a refinement of that inequality.

\begin{theorem}[Main result] \label{main} If $F: D \rightarrow \mathbb{C}$ is holomorphic on some neighborhood of the unit disk and has a Blaschke factorization $F = B\cdot G$, then
$$ \|G(e^{i \cdot})\|_{X} \leq \|F(e^{i \cdot})\|_{X}.$$
Moreover, if $F(\alpha)=0$ for some $\alpha \in \mathbb{D}$, we even have
$$ \|G(e^{i \cdot})\|^2_{X} \leq \|F(e^{i \cdot})\|^2_{X} - (1-|\alpha|^2)\left\| \frac{G(e^{i \cdot})}{1 - \overline{\alpha}z}\right\|^2_{Y}.$$
\end{theorem}

The most important implication is convergence of the
unwinding series in the space $Y$ if the initial data lies in $X$. The argument is straight-forward: the construction of the unwinding series proceeds by setting
$$ B_{n+1}(z) G_{n+1}(z) = G_n(z) - G_n(0)$$
and thus, by construction, the functions always have a root in $\alpha = 0$. Furthermore, adding and subtracting constants has no impact
on $\left\|\cdot \right\|_{X}$ because $\gamma_0 = 0$ and therefore
\begin{align*}
\|G_{n+1}(e^{i \cdot})\|^2_{Y} &\leq \|G_{n}(e^{i \cdot})-G_{n}(0)\|^2_{X} - \|G_{n+1}(e^{i \cdot})\|^2_{X} \\
&= \|G_{n}(e^{i \cdot})\|^2_{X} - \|G_{n+1}(e^{i \cdot})\|^2_{X}
\end{align*}
Summing on both sides yields a telescoping series and thus
$$ \sum_{n=2}^{\infty}{\|G_n(e^{i \cdot})\|^2_{Y}} \leq \|F(e^{i \cdot})\|^2_{X},$$
which implies that $\|G_n(e^{i \cdot})\|_{Y} \rightarrow 0$. After $n$ steps, we have the equation
$$ F = a_1 B_1 + a_2 B_1 B_2 + \dots +  a_{n-1} B_1 \cdots B_{n-1} + B_1 B_2 \cdots B_{n-1} (G_n-G_n(0))$$
and exploiting that $|B_i| = 1$, we have that
$$ \| f - (a_1 B_1 + \dots + a_{n-1} B_1 \cdots B_{n-1})\|_{L^2(\partial \mathbb{D})} = \|G_n-G_n(0)\|_{L^2(\partial \mathbb{D})}.$$
This motivates putting special emphasis on the space $X$ arising from $\gamma_n = n$ for which $Y = L^2$. This
space is also known as the Dirichlet space $\mathcal{D}$ and has special geometric significance and structure; 
for algebraic reasons we can get an even sharper inequality in that case (see below). 
Another natural (from a geometric perspective) space is given by $\gamma_n = n^2$, where $X = H^1$, $Y = \mathcal{D}$ and 
Theorem 1 can be alternatively proven using Green's formula (see below). 
All Sobolev spaces $H^s$ with $s > 0$ are also special cases: the statement implies that for $F(e^{i \cdot}) \in H^s$ with $s > 0$,
we have convergence in $H^{s-\frac{1}{2}}$. 
All these results have a completely analogous version on the upper half-space $\mathbb{C}_{+}$ with Blaschke-type
products being defined on the real line $\mathbb{R}$; even the proofs translate almost verbatim (see below).

\subsection{A slight generalization.} The unwinding series can be phrased slightly more generally than we have done up to now: indeed, at the $n-$th
step, we could actually pick an arbitrary $\alpha_n \in \mathbb{D}$ and proceed via
$$ B_{n+1}(z) \cdot G_{n+1}(z) = G_n(z) - G_n(\alpha_n).$$
Clearly, $G_n(z) - G_n(\alpha_n)$ is guaranteed to have at least one root in the unit disk because it has one in $\alpha_n \in \mathbb{D}$.
Theorem 1 was formulated in a completely general way (for a general root $\alpha$) and applies to this more general case as well (at the
cost of introducing a factor $1-|\alpha_n|^2$). In choosing a `good' value for $\alpha_n$, one naturally encounters the quantity
$$ \arg\max_{z \in \mathbb{D}}{ (1-|z|^2)|G(z)|}.$$
We set $\alpha_n = 0$, which -- in practice -- does not seem to make a big difference because the factor $(1-|z|^2)$ ensures that the
maximum cannot be assumed on the boundary. While maximizing the quantity can lead to better results, we have observed that
 $\alpha_n = 0$ seems to always be doing fairly well in practice. We will assume $\alpha_n = 0$ throughout the rest of the paper but 
emphasize that the algorithm is slightly more general. One instance where this could be useful is whenever $G_n(z)-G_n(0)$ has a
root on $\partial \mathbb{D}$: in order not to lose information on the phase, it is desirable for the performance of the Guido \& Mary Weiss
algorithm that $|G_n(z)-G_n(0)| > 0$ for all $|z| = 1$.
Whenever this is not the case, one could use $G_n(z) - G_n(\alpha)$ for a value $\alpha \in \mathbb{D}$ close to
the origin such that
this function has no roots on the boundary.

\subsection{A special case.} Let us now explore the special cases with obvious geometric significance in greater
detail. We identify functions $F$ that are holomorphic in a neighborhood of the unit disk with maps $\gamma: \mathbb{T} \rightarrow \mathbb{R}^2$ via
$$ \gamma_F(t) := F(e^{it}).$$ This is motivated by the fact that in the algorithm we obtain $G_{n+1}$ not from $G_n(z)$ but from
$G_n(z) - G_n(0)$ and it is therefore natural to study translation-invariant (geometric) quantities depending on $G_n$.
Let us consider a particular example $F(z) = (z+0.3+i/3)(z-0.2)(z-1.5-i/2)$ (taken essentially at random) and the Blaschke factorization $F(z) = B(z) G(z)$ (see Fig. 2 and Fig. 3).
Since $G(z)$ has no roots in $\mathbb{D}$, the argument principle implies that $G(z)$ does not wind around 0. Note that furthermore 
$$|F(e^{it})| = |G(e^{it})|$$ for all $t \in \mathbb{R}$.
\begin{figure}[h!]
\includegraphics[width = 0.5\textwidth]{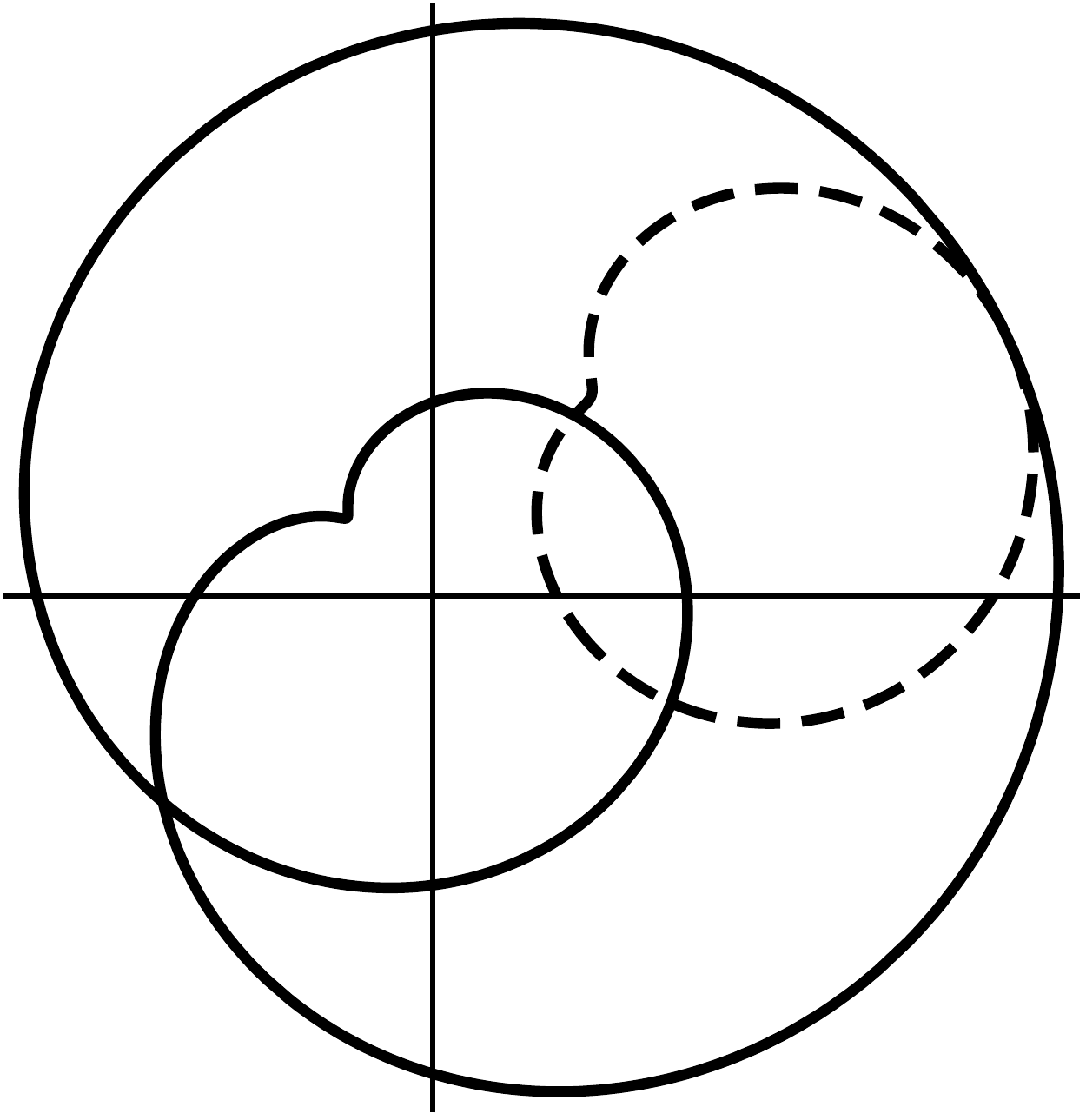}
\caption{$F(e^{it})$ and $G(e^{it})$ (dashed) for a cubic polynomial $F$.}
\end{figure}
As suggested by the picture (and many others like it), one would 
expect that the length of the curve $\gamma_{G}$ is, at least generically, smaller than that of $\gamma_{F}$ but we have been unable 
to prove that; instead, we were able to obtain that result for the natural $L^2-$version of length, sometimes called the \textit{energy} of a curve
$$ \mbox{energy}(\gamma) = \int_{0}^{2\pi}{|\gamma'(t)|^2dt}.$$
By H\"older's inequality, we have that
$$ \mbox{length}(\gamma)^2 = \left(  \int_{0}^{2\pi}{|\gamma'(t)|dt} \right)^2 \leq 2\pi  \left(  \int_{0}^{2\pi}{|\gamma'(t)|^2dt} \right) = 2\pi \cdot \mbox{energy}(\gamma).$$
Therefore, in particular, if the energy of a curve tends to 0, then so will the length. Algebraic simplifications allow us to quantify the decrease of the $H^1-$norm of
the boundary function in terms of its $L^2-$norm weighted against the Poisson kernel of the roots: the argument is not as sharp as the one formulated for the
Dirichlet space further below but is very elementary (using Green's theorem and geometric considerations).

\begin{theorem}\label{fullformula} Let $F:D \rightarrow \mathbb{C}$ be holomorphic in some neighborhood of the unit disk. Then,
if $\left\{\alpha_j:j \in J\right\}$ are the roots of $F$ in $\mathbb{D}$ and $ F = B \cdot G,$ 
$$ \int_{0}^{2\pi}{| G'(e^{i\theta})|^2 dt} \leq  \int_{0}^{2\pi}{| F'(e^{i\theta})|^2 dt}  - \int_{0}^{2\pi}{|G(e^{it})|^2 \sum_{j \in J}{\frac{1-|\alpha_j|^2}{|e^{it}-\alpha_j|^2}}} dt.$$
\end{theorem}

Exploiting an additional geometric argument based on random projections and the uncertainty principle, we
were able to obtain the following estimate, which controls the error in $L^{\infty}(\mathbb{T})$.
\begin{corollary}\label{fullderiv} Suppose $F:D \rightarrow \mathbb{C}$ converges on some neighborhood of the unit disk. Then
the formal series converges in $L^{\infty}$. Moreover,
$$ \left|\left\{n \in \mathbb{N}: \|G_n(z) - G_n(0)\|_{L^{\infty}(\partial \mathbb{D})} \geq \varepsilon \right\}\right| \lesssim \left(  \int_{0}^{2\pi}{|F'(e^{it})|^2 dt }\right)^2/\varepsilon^4.$$
\end{corollary}

\subsection{Winding numbers and the Dirichlet space.} This section is entirely motivated by geometric considerations: we will discuss properties of closed curves in $\mathbb{C}$ given by
$\gamma_F(t) = F(e^{it})$.
The winding number around a point $z$ with respect to a curve $\gamma$ is defined as
$$ \wind_{\gamma}(z_0) := \frac{1}{2\pi i} \int_{\gamma}{ \frac{dz}{z - z_0}} \qquad \mbox{whenever}~z_0 \notin \gamma.$$
Examples strongly suggest that 'the \textit{average} weighted winding number'
$$ \int_{\mathbb{C}}{\wind_{\gamma}(z)dz} \qquad \mbox{should decrease.}$$
This quantity can be regarded as \textit{weighted} area, which is the area enclosed by the curve weighted
with the winding number. It arises naturally when one applies Green's formula to compute the area surrounded by a simple, closed curve
$\gamma:[0,2\pi] \rightarrow \mathbb{R}^2$ oriented counter-clockwise and written as $\gamma(t) = (x(t), y(t))$ via
$$ \frac{1}{2}\int_{0}^{2\pi}{(x(t)\dot y(t) - \dot{x}(t)y(t)) dt}.$$
Applying the very same formula in the case of a non-simple closed curve naturally gives rise to
$$ \frac{1}{2}\int_{0}^{2\pi}{(x(t)\dot y(t) - \dot{x}(t)y(t)) dt} = \int_{\mathbb{C}}{\wind_{\gamma}(z)dz}.$$
This interpretation of the area formula dates back at least to a 1936 paper of Rado \cite{rado}. 
If $F$ is holomorphic, then we have 
$$ \int_{\mathbb{C}}{\wind_{\gamma_F}(z)dz} = \int_{\mathbb{D}}{|F'(z)|^2dz}.$$

\begin{figure}[h!]
\includegraphics[width = 0.5\textwidth]{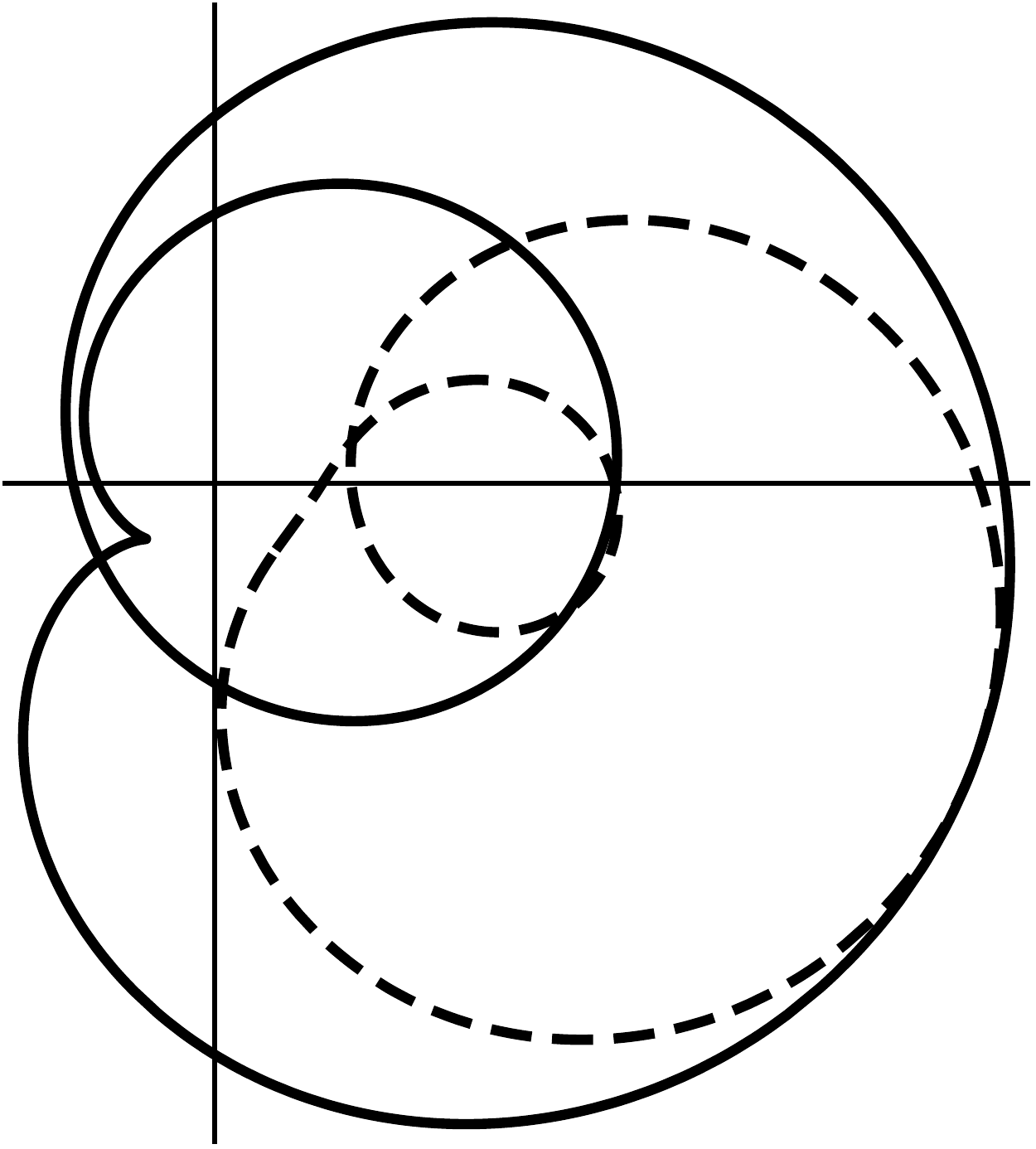}
\caption{$F(e^{it})$ (blue) given by the cubic polynomial. $G(e^{it})$ (dashed) has the same maximum winding but over a smaller area.}
\end{figure}


Writing that representation in Fourier space gives the so-called \textit{area theorem} stating that if
$$ f(z) = a_0 + a_1z + a_2z^2 + \dots, \quad \mbox{then} ~  \int_{\mathbb{C}}{\wind_{\gamma_F}(z)dz} = \pi\sum_{n=1}^{\infty}{n|a_n|^2}.$$
The \textit{Dirichlet space} 
$$\mathcal{D} = \left\{f:\mathbb{D} \rightarrow \mathbb{C}\big| ~ f~\mbox{holomorphic}~\mbox{and}~\int_{\mathbb{D}}{|f'(z)|^2dz} < \infty\right\}.$$
was first introduced by Beurling and Deny \cite{beur, beur2} .When equipped with the inner product
$$ \left\langle f, g \right\rangle_{\mathcal{D}} = \left\langle f, g \right\rangle_{\mathcal{H}^2} + \frac{1}{\pi}\int_{\mathbb{D}}{f'(z)\overline{g'(z)}dz},$$
it becomes a Hilbert space.
A monotonicity statement for Blaschke decomposition in that space is well-known and follows at once from Carleson's formula \cite{ca} (see also \cite[Theorem 4.1.3]{primer}). 

\begin{corollary}[Special case of Carleson's formula] \label{halfformula} Assume $F \in \mathcal{H}^{\infty}(\mathbb{D})$ with roots $\left\{\alpha_j:j\in J\right\}$ in $\mathbb{D}$ and has the Blaschke factorization $ F = B \cdot G$, then
$$ \int_{\mathbb{D}}{|F'(z)|^2dz} =  \int_{\mathbb{D}}{|G'(z)|^2dz}+ \frac{1}{2}\int_{\partial \mathbb{D}}{|G|^2 \sum_{j \in J}{\frac{1-|\alpha_j|^2}{|z-\alpha_j|^2}}}.$$
\end{corollary}
This result is better than Theorem 1 (which only gives the constant 1 instead of the sum over the Poisson kernel indexed by the roots) but follows from the
same argument that we use to prove Theorem 1. This is due to some
algebraic simplification that seems to only occur for $X = \mathcal{D}, Y = L^2$ and has to do with the fact that for $\alpha \in \mathbb{D}$
$$|z - \alpha| = |1-\overline{\alpha}z| \qquad \mbox{whenever}~|z|=1.$$ 

\subsection{A curious stability property.} 
When doing Blaschke factorization $F = B\cdot G$ numerically, we will introduce some roundoff errors; even though we never actually compute the roots of the functions, this roundoff
error can be imagined as perturbing the roots a little bit. We have the following curious and purely algebraic \textit{pointwise} stability statement.
\begin{theorem} \label{rootstab} Suppose $F_1, F_2:\mathbb{C} \rightarrow \mathbb{C}$ are polynomials having the same roots outside of $\mathbb{D}$ and the same number of roots inside $\mathbb{D}$.
Then the Blaschke factorizations
$$ F_1 = B_1G_1 \qquad \mbox{and} \qquad F_2 = B_2G_2,$$
satisfy
$$ |G_1(z) - G_2(z)| = |F_1(z) - F_2(z)| \qquad \mbox{whenever}\quad |z| =1.$$
\end{theorem}
This stability property was discovered by accident and seems quite curious. It is not clear to us whether there might be even more general
statements of a similar type.

\subsection{An unwinding series on $\mathbb{R}$.} The inner-outer factorization was the crucial ingredient to our entire approach. A similar
factorization can be achieved on the upper half-space. 
The role of Blaschke products is now played by functions indexed by $\lambda_1, \dots, \lambda_n \in \mathbb{C}_{+}$ of the form
$$B(z) = \prod_{k=1}^{n}{ \frac{z-\lambda_k}{z-\overline{\lambda_k}}}, \qquad \mbox{which satisfies}~|B(z)| = 1~\mbox{ on}~ \mathbb{R}.$$
We will consider norms $\left\| \cdot \right\|_{X}, \left\| \cdot \right\|_{Y}$ on the space
$$ L^2_{+}(\mathbb{R}) = \left\{f \in L^2(\mathbb{R}): \supp(\widehat{f}) \subseteq [0, \infty) \right\}.$$
Let $\psi:[0,\infty] \rightarrow [0,\infty]$ be a monotonically increasing, differentiable function with $\psi(0) = 0$ and
$$ \|f\|^2_{X} := \int_{0}^{\infty}{|\widehat{F}(\xi)|^2 \psi(\xi) d\xi} \qquad \mbox{as well as} \qquad  \|f\|^2_{Y} := \int_{0}^{\infty}{|\widehat{F}(\xi)|^2 \psi'(\xi) d\xi}.$$
\begin{theorem} \label{complex} If $F$ has roots $\lambda_1, \dots, \lambda_n \in \mathbb{C}_{+}$, then
$$  \left\|  F \prod_{i=1}^{n}{\frac{z-\overline{\lambda_k}}{z-\lambda_k}}   \right\|^2_{X} \leq  \|F\|^2_{X}.$$
For the removal of a single root $F(\lambda) = 0$, we have the stronger estimate
$$  \left\|F \frac{z-\overline{\lambda}}{z-\lambda}\right\|^2_{X} \leq  \|F\|^2_{X} - (2 \Im( \lambda ))\left\|F \frac{z-\overline{\lambda}}{z-\lambda}\right\|^2_{Y}.$$
Moreover, in the Dirichlet space $\psi(\xi) = \xi$, we even have
$$  \left\|  F \prod_{i=1}^{n}{\frac{z-\overline{\lambda_k}}{z-\lambda_k}}   \right\|^2_{X} \leq  \|F\|^2_{X} - \int_{\mathbb{R}}{|F(x)|^2 \sum_{i=1}^{n}{\frac{2 \Im( \alpha) }{|x-\lambda_k|^2}} dx},$$
where the sum ranges over all roots of $F$ on $\mathbb{C}_{+}$.
\end{theorem}

\section{Computation and application}
In this section we provide a collection of known facts, additional background material, a way of computing the Blaschke factorization without ever having to compute the roots (dating back to a 1962 paper of Guido and Mary Weiss) and some sample applications.
\subsection{Analytic signals.} A classical way of using complex analysis when faced with a periodic, real signal $u(t):[0,2\pi] \rightarrow \mathbb{R}$ 
is to associate a natural imaginary part to the function. Already in 1946 Gabor  \cite{gabor} argued that
\begin{quote}
it has long been recognized that operations with the complex exponential $e^{j \omega t}$ [...] have distinct advantages over operations
with sine or cosine functions.
\end{quote}
and proposed to analyze the signal
$$ f = u + i\mathcal{H}_{i}u \qquad \mbox{instead,}$$
where $\mathcal{H}_{i}$ is the Hilbert transform. Vakman \cite{vakman} proved that requiring certain natural assumptions on the complexification process,
this is the canonical complexification. A convenient fact for actual computation is that if
$$ u(\theta) = \frac{a_0}{2} + \sum_{k \geq 1}{a_k \cos{k \theta} + b_k \sin{k \theta}},$$
then
$$ (\mathcal{H}_{i}u)(\theta) = \sum_{k \geq 1}{a_k \sin{k \theta} - b_k \cos{k \theta}} \quad \mbox{and} \quad  (u+i\mathcal{H}u)(\theta) = \frac{a_0}{2}+\sum_{k \geq 1}{(a_k - i b_k)e^{i k \theta}}.$$

\subsection{The Guido and Mary Weiss algorithm.} Let now $f(\theta)$ be a complex signal (possibly obtained from a real signal using the process
above). Assume additionally that $f(\theta) \neq 0$. Note that any such $f(\theta)$ has only positive frequencies
$$ f(\theta) = \sum_{k=0}^{\infty}{a_k e^{i k \theta}}$$
to which we may associate the function $F:\mathbb{D} \rightarrow \mathbb{C}$
$$ F(z) = \sum_{k = 0}^{\infty}{a_k z^k},$$
which, assuming sufficient regularity, has $f$ as its boundary function. It is now our goal to construct the Blaschke decomposition of
$ F = B \cdot G$ without computing the roots of the function.\\
 The algorithm proceeds as follows.
\begin{quote}
\begin{enumerate}
\item Compute the function $g(\theta) = \log{|f(\theta)|}$.
\item Compute the analytic signal from $g$
$$ h(\theta) = (g + i\mathcal{H}_{i}g)(\theta).$$
\item Then we have the Blaschke factorization $F = B \cdot G$, where
$$ G(\theta) = e^{h(\theta)} \quad \mbox{and} \quad B(\theta) = F(\theta)/G(\theta)$$
on the unit circle.
\end{enumerate}
\end{quote}
Clearly, the algorithm won't work whenever there is a root on the boundary of the unit disk because
$\log{|f(\theta)|}$ will be unbounded; also, whenever $|f(\theta)|$ becomes very small, the algorithm
becomes unstable. Various ways for additional stabilization have been proposed: the stabilizing effect
of adding a small constant has been investigated by Nahon \cite{nahon} whereas Letelier and Saito \cite{let}
propose adding a small pure sinusoid.

\subsection{Removal of multiplicative noise.} We return to the analogy
$$B \sim \mbox{frequency and}~G \sim \mbox{amplitude}.$$
$B$ is constructed from $F$ by its roots inside the unit disk; conversely, given $F(e^{i\cdot})$ as boundary
data, we can uniquely reconstruct the values of $F$ inside $\mathbb{D}$ by convolving with the boundary
data with the Poisson kernel. This compact integral operator enjoys a variety of smoothing properties; as
a consequence it is stable against all sorts of perturbations (assuming they roughly preserve the local
averages).
\begin{figure}[h!]
\includegraphics[width = 0.6\textwidth]{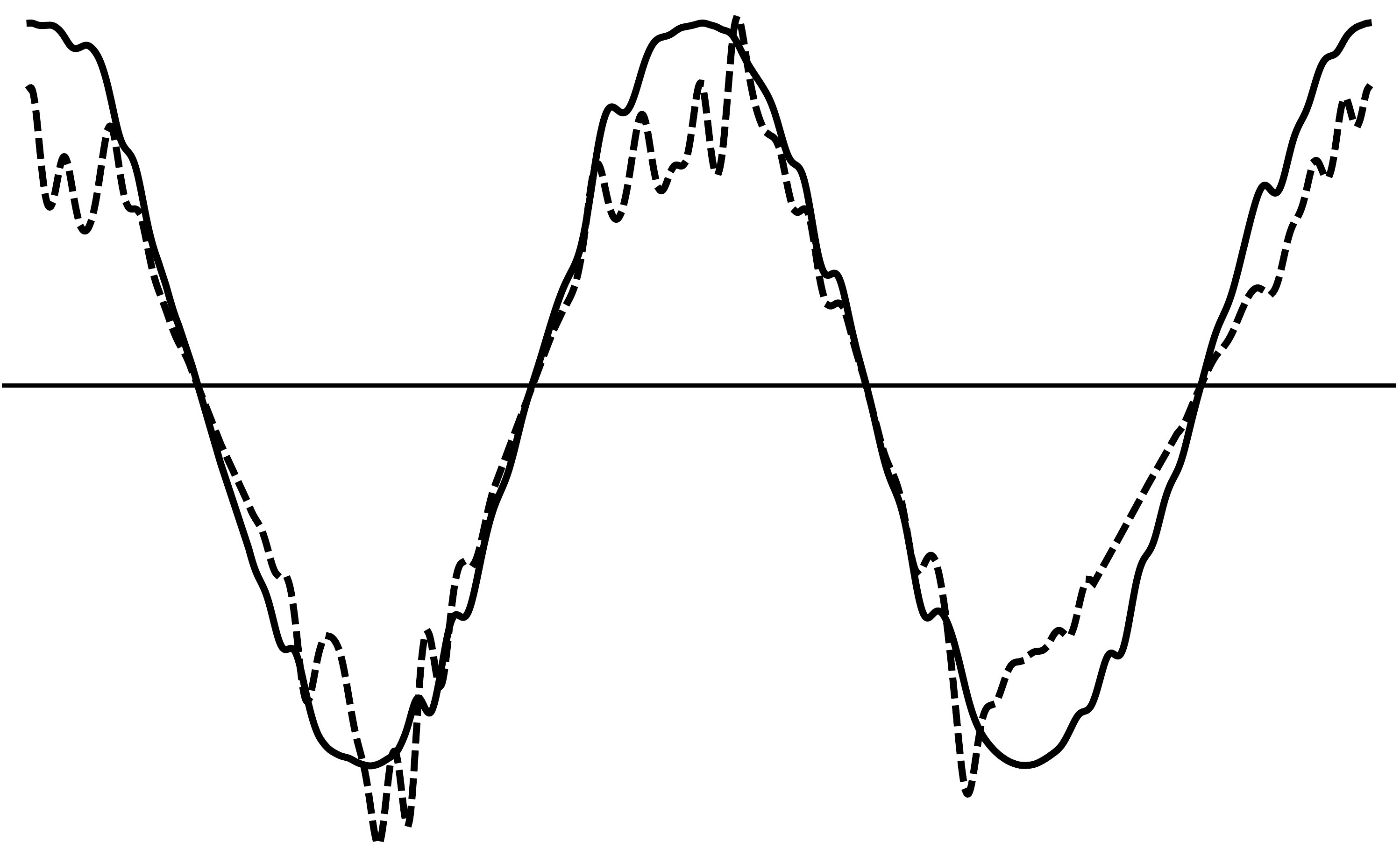}
\caption{An example of a multiplicatively perturbed pure frequency $\cos{(2\theta)}$ (dashed) on $\mathbb{T}$ and the resulting curve after
two rounds of filtering.}
\end{figure}
A particular example given in (Fig. 4) consists of a function of the type
$$ f(\theta) = \cos{(2\theta)} \left( \sum_{n=1}^{50}{\alpha_n\frac{\cos{n\theta}}{\sqrt{n}} + \beta_n\frac{\sin{n\theta}}{\sqrt{n}}} \right),$$
where $\alpha, \beta$ are instances of i.i.d. $\mathcal{N}(0,1)$ random variables. We complexify the signal $F$ and replace it by
$F(\theta)/|F(\theta)|$. The outcome of two iterations of this process is shown in Fig. 4.
A similar example can be found in work of Letelier \& Saito \cite{let} and Healy \cite{healy, healy2}.

\subsection{The instanteous phase problem.} Given a complex signal, we may write
it in polar coordinates as
$$ F(\theta) = |F(\theta)|e^{i\phi(\theta)}.$$
It is of interest in practice to understand how fast the frequency changes; naturally, if all quantities are well defined,
$$ F' = |F|'e^{i \phi} + i \phi' |F|e^{i \phi}$$
and thus
$$ \frac{F'}{F} = \frac{|F|'}{|F|} + i \phi'.$$
However, even assuming sufficient smoothness, the direct computation of the instaneneous frequency via this identity can be challenging and numerically unstable; various
methods have been proposed (including one using Blaschke products due to Picinbono \cite{pic}).
Blaschke products are well known to have the following particularly nice property: if the Blaschke product has
finitely many roots, then
$$  e^{im \theta}\prod_{k}{\frac{\overline{\alpha_k}}{|a_k|}\frac{e^{i\theta}-\alpha_k}{1-\overline{\alpha_k}e^{i \theta}}} = |F(\theta)|e^{i \phi(\theta)},$$
and one has
$$ \phi'(\theta) = m + \sum_{k }{\frac{1-|\alpha_k|^2}{|e^{i \theta} - \alpha_k|^2}} > m \geq 0.$$
The unwinding series is therefore an approximation using strictly increasing frequencies, which 
greatly stabilizes numerical computation (see \cite{nahon} for details).

\section{Proofs}
\subsection{T. Qian's Theorem.} We start by giving a brief summary and proof of T. Qian's theorem. This material is not new
and can be found in \cite{qtao5}, however, that paper may not be easily accessible. 

\begin{thm}[T. Qian, \cite{qtao5}] The unwinding series converges in $L^2$ for all
$$ F(\theta) = \sum_{n \geq 0}{a_n e^{i n \theta}} \quad \mbox{with} \quad \sum_{n=0}^{\infty}{|a_n|^2} < \infty.$$
\end{thm}
\begin{proof}[Proof from \cite{qtao5}] We first write the unwinding series in a slightly different way: since $0$ is always a root in the iteration scheme,
we may write it as
$$ F(z) = F(0) + \gamma_1 z B_1 + \gamma_2 z^2 B_1 B_2 + \gamma_3 z^3 B_1 B_2 B_3 + \dots + z^n B_1 B_2 \dots B_n (G-G(0)).$$
We  remark that any two of the Blaschke terms are orthogonal on $L^2(\mathbb{T})$: if $\ell < m$, then
$$ \int_{0}^{2\pi}{ \overline{\gamma_{\ell} e^{i t \ell} \left(\prod_{k=1}^{\ell}{B_k(e^{it})}\right)}  \gamma_m e^{i t m} \left(\prod_{k=1}^{m}{B_k(e^{it})}\right) dt} =  \int_{0}^{2\pi}{ \overline{\gamma_{\ell}}  \gamma_m e^{i t (m-\ell)} \left(\prod_{k=\ell+1}^{m}{B_k(e^{it})}\right) dt} = 0$$
because $|B_k(e^{it})| = 1$ and the remaining term is holomorphic. We furthermore observe that the last term is orthogonal to all previous terms since the inner product simplifies by the same computation to
$$ \int_{0}^{2\pi}{ \overline{\gamma_{\ell}}  \gamma_n e^{i t (n-\ell)} \left(\prod_{k=\ell+1}^{n}{B_k(e^{it})}\right) (G(e^{it}) - G(0)) dt} = 0$$
 This immediately implies that
\begin{align*}
 \|F(e^{it})\|^2_{L^2(\mathbb{T})} &=  \|F(0)\|^2_{L^2(\mathbb{T})} +  \left\|\gamma_1 e^{it} B_1(e^{it})\right\|^2_{L^2(\mathbb{T})}  + \dots\\
&+  \left\|e^{i n t}B_1(e^{it}) B_2(e^{it}) \dots B_n(e^{it})(G(e^{it})-G(0))  \right\|^2_{L^2(\mathbb{T})}.
\end{align*}
However, we can also guarantee that the remainder term is small by showing that it is orthogonal to all $\left\{z^k: 0 \leq k \leq n-1\right\}$ since
$$ \int_{0}^{2\pi}{ \gamma_n e^{i t n} \left(\prod_{k=1}^{n}{B_k(e^{it})}\right) (G(e^{it})-G(0)) e^{-i k t} dt} = 0.$$
This implies
$$  \left\|e^{i n t}B_1(e^{it}) B_2(e^{it}) \dots B_n(e^{it}) (G(e^{it})-G(0))  \right\|^2_{L^2(\mathbb{T})} \leq \sum_{k=n}^{\infty}{|a_k|^2},$$
which then implies convergence as $n \rightarrow \infty$.
\end{proof}
The proof shows that convergence will happen at least as quickly as Fourier series but potentially much faster since
low-lying terms can already contain some part of the high-frequency contributions. It would be interesting to quantifying how precisely this happens.

\subsection{Proof of Theorem \ref{main}.}
 We study the action of moving a single root from inside the unit disk $\mathbb{D}$ to the outside (inversion along the unit circle). Let $|\alpha| < 1$
be the root; we compare
$$ f(z) = (z-\alpha)F(z) \qquad \mbox{and} \qquad g(z) = (1-\overline{\alpha}z)F(z)$$
on the boundary $\partial \mathbb{D}$.
Expanding $F\big|_{\partial \mathbb{D}}$ into a Fourier series
$$ F(z)\big|_{\partial \mathbb{D}} = \sum_{n=0}^{\infty}{a_n z^n},$$
we immediately get
\begin{align*}
 f(z)\big|_{\partial \mathbb{D}} = -\alpha a_0 + \sum_{n=1}^{\infty}{(a_{n-1}-\alpha a_n) z^n} \quad \mbox{and} \quad g(z)\big|_{\partial \mathbb{D}} &=  a_0 + \sum_{n=1}^{\infty}{(a_{n}-\overline{\alpha} a_{n-1}) z^n}.
\end{align*}
From the definition of $\|\cdot\|_{X}$, we compute
\begin{align*}
  \|f(z)\big|_{\partial \mathbb{D}}\|^2_{X} &= \gamma_0 |\alpha|^2 |a_0|^2 + \sum_{n=1}^{\infty}{ \gamma_n |a_{n-1}-\alpha a_n|^2}  \\
&= \gamma_0 |\alpha|^2 |a_0|^2 + \sum_{n=1}^{\infty}{ \gamma_n \left( |a_{n-1}|^2 - \overline{\alpha} a_{n-1} \overline{a_n} - \alpha a_n \overline{a_{n-1}} + |\alpha|^2 |a_n|^2\right)}
\end{align*}
and
\begin{align*}
  \|g(z)\big|_{\partial \mathbb{D}}\|^2_{X} &= \gamma_0 |a_0|^2 + \sum_{n=1}^{\infty}{ \gamma_n |a_{n}-\overline \alpha a_{n-1}|^2}  \\
&= \gamma_0 |a_0|^2 + \sum_{n=1}^{\infty}{ \gamma_n \left( |a_{n}|^2 - \overline{\alpha} a_{n-1} \overline{a_n} - \alpha a_n \overline{a_{n-1}} + |\alpha|^2 |a_{n-1}|^2\right)}.
\end{align*}
We see that the mixed terms appear in both sums and cancel: subtraction yields
\begin{align*} \|f(z)\big|_{\partial \mathbb{D}}\|^2_{X} - \|g(z)\big|_{\partial \mathbb{D}}\|^2_{X} &= -\gamma_0 (1-|\alpha|^2)|a_0|^2 + (1-|\alpha|^2)\sum_{n=1}^{\infty}{\gamma_n(\left|a_{n-1}\right|^2 - \left|a_n\right|^2)}\\
&= (1-|\alpha|^2)\sum_{n=0}^{\infty}{(\gamma_{n+1}-\gamma_{n})\left|a_{n}\right|^2}\\
&= (1-|\alpha|^2)\|F\big|_{\partial \mathbb{D}}\|^2_{Y}.
\end{align*}
This equation has a nice and definite form but we will only use it in one instance. Let us assume we are given $F(Z)$ and a finite list of roots $\left\{\alpha_1, \alpha_2, \dots, \alpha_n\right\} \subset \mathbb{D}$.
We know, by construction, that at least one of the roots is 0 and we assume without loss of generality that $\alpha_n = 0$. Then we can consider the sequence of functions
\begin{align*}
F(z) &= (z-\alpha_1)(z-\alpha_2)(z-\alpha_3)\dots (z-\alpha_{n-1})  (z-\alpha_n) H(z) \\
F_1(z) &= (1-\overline{\alpha_1}z)(z-\alpha_2)(z-\alpha_3)\dots  (z-\alpha_{n-1})  (z-\alpha_n) H(z) \\
F_2(z) &= (1-\overline{\alpha_1}z)(1-\overline{\alpha_2}z)(z-\alpha_3)\dots  (z-\alpha_{n-1})  (z-\alpha_n) H(z) \\
\dots \\
F_{n-1}(z)&= (1-\overline{\alpha_1}z)(1-\overline{\alpha_2}z)(1-\overline{\alpha_3}z)\dots  (1-\overline{\alpha_{n-1}}z)  (z-\alpha_n) H(z) 
\end{align*}
and we can conclude from the computation that
$$ \|F(z)\big|_{\partial \mathbb{D}}\|_{X} \geq  \|F_1(z)\big|_{\partial \mathbb{D}}\|_{X} \geq \dots \geq  \|F_{n-1}(z)\big|_{\partial \mathbb{D}}\|_{X}.$$
Clearly, the outer function $G$ in the Blaschke decomposition $F = B \cdot G$ will be given by
$$G(z) = (1-\overline{\alpha_1}z)(1-\overline{\alpha_2}z)(1-\overline{\alpha_3}z)\dots  (1-\overline{\alpha_{n-1}}z)  (1-\overline{\alpha_{n}}z) H(z) $$
In the final step, we use the fact that there is always one root satisfying $\alpha_n = 0$ and exploit the full strength of the argument to conclude that
\begin{align*} \|F_{n-1}(z) \big|_{\partial \mathbb{D}}  \|^2_{X} -  \|F_{n}(z) \big|_{\partial \mathbb{D}}\|^2_{X} = \left\|  \prod_{k=1}^{n-1}{(1-\overline{\alpha_k}z)} H(z)   \big|_{\partial \mathbb{D}}\right\|^2_{Y} = \|  G(z)   \big|_{\partial \mathbb{D}}\|^2_{Y}.
\end{align*}
More, generally, if there is no root in 0, then applying the same argument yields
\begin{align*} \|F_{n-1}(z) \big|_{\partial \mathbb{D}}  \|^2_{X} -  \|F_{n}(z) \big|_{\partial \mathbb{D}}\|^2_{X} &= (1-|\alpha_n|^2) \left\|  \prod_{k=1}^{n-1}{(1-\overline{\alpha_k}z)} H(z)   
\big|_{\partial \mathbb{D}}\right\|^2_{Y} \\
&= (1-|\alpha_n|^2)  \left\|  \frac{G(z)}{1-\overline{\alpha_n}z}   \big|_{\partial \mathbb{D}} \right\|^2_{Y}.
\end{align*}
This concludes the argument. $\qed$\\

The last part of the argument highlights a fundamental difficulty: while there is an effective gain every time we move a
root to the outside, it is not clear to us how the sum of these gains could be properly controlled (which is why we only take the last one). This we only managed
to do in the case of the Dirichlet space, where an additional (algebraic) simplification takes place.

\subsection{Proof of Theorem \ref{fullformula}.} We study again the action of moving a single root to the outside by inversion along the unit circle.
The computation resembles the computation in the more general case except that we are able to invoke Green's formula
at the end of the argument.
\begin{lemma}
Let $F$ be analytic in a neighborhood of the origin and $a \in \mathbb{C}$ with $|a| < 1$. If
$$ f = (z-a)F \qquad \mbox{and} \qquad g = (1-\overline{a}z)F$$
then
$$  \int_{0}^{2\pi}{| g'(e^{i\theta})|^2 dt} \leq \int_{0}^{2\pi}{| f'(e^{i\theta})|^2 dt} - (1-|a|^2)\int_{0}^{2\pi}{|F(e^{it})|^2dt}$$
whenever all terms are defined and finite.
\end{lemma}

\begin{proof} 

Obviously
$$ f' = F + (z-a)F'$$
and thus
$$ |f'|^2 = |F|^2 + F\overline{(z-a)F'} + \overline{F}(z-a)F' + |z-a|^2|F'|^2.$$
At the same time
\begin{align*} 
g' &= -\overline{a}F + (1-\overline{a}z)F' \\
 |g'|^2 &= |a|^2|F|^2 - a\overline{F}(1-\overline{a}z)F' - \overline{a}F\overline{(1-\overline{a}z)F'} +|1-\overline{a}z|^2|F'|^2.
\end{align*}
If $|z| = 1$, then $|z-a|^2 = |1-\overline{a}z|^2$ and since we only integrate over $\partial \mathbb{D}$, we get
$$ \int_{\partial \mathbb{D}}{  |f'|^2  -  |g'|^2 } = (1-|a|^2)\int_{\partial \mathbb{D}}{|F|^2} \\
+ \int_{\partial \mathbb{D}}{ F\overline{(z-a)F'} + \overline{F}(z-a)F'  + a\overline{ F}(1-\overline{a}z)F'  + \overline{a}F\overline{(1-\overline{a}z)F'}   }.$$
This is already almost what we want, it remains to show that

$$  \int_{\partial \mathbb{D}}{ F\overline{(z-a)F'} + \overline{F}(z-a)F'  + a\overline{F}(1-\overline{a}z)F'  +\overline{a}F\overline{(1-\overline{a}z)F'}}   \geq 0.$$
The expression can be rewritten as
$$  \int_{\partial \mathbb{D}}{ F\overline{F'} ( \overline{(z-a)} +\overline{a}\overline{(1-\overline{a}z)}) + \overline{F}F' ((z-a) + a(1-\overline{a}z)) },$$
which is 
 $$  (1-|a|^2)\int_{\partial \mathbb{D}}{ F\overline{F'} \overline{z} + \overline{F}F'z},$$
Now we go back from the classical derivative $f'(z)$ to the angular derivative along the boundary of the disk $\dot{f}(z)$. As before
$$\frac{d}{d\theta}f(e^{i\theta}) = f'(e^{i\theta})e^{i\theta}i$$
which can be rewritten as
$$ f'(z) = -i\overline{z} \dot{f}(z) \qquad \mbox{whenever}~|z| = 1.$$
Using this, we can rewrite the terms as
$$  F \overline{z} \overline{F'} = F \overline{z} \overline{(-i)}z \overline{\dot{F}} = i F  \overline{\dot{F}}   \qquad \mbox{whenever}~|z| = 1$$
$$  F' \overline{z} \overline{F} = - i z\overline{z} \dot{F} \overline{F} = - i \dot{F}  \overline{F}   \qquad \mbox{whenever}~|z| = 1.$$
We need to show that
$$ i\int_{\partial \mathbb{D}}{F  \overline{\dot{F}}  -\dot{F}  \overline{F}   } \geq 0.$$
If we write
$$ F(e^{it}) = x(t) + i y(t),$$
then
$$i (F  \overline{\dot{F}}  -\dot{F}  \overline{F}) = 2(x(t) \dot{y}(t) - \dot{x}(t) y(t)).$$
The problem consists now of evaluating
$$ \int_{\partial \mathbb{D}}{2(x(t) \dot{y}(t) - \dot{x}(t) y(t)) dt}.$$
This corresponds to integrating the vector field
$$ 2(-y, x) \qquad \mbox{over the curve} \qquad \gamma(t) = (x(t), y(t)).$$
Green's theorem states that this implies
$$ \int_{\partial \mathbb{D}}{2(x(t) \dot{y}(t) - \dot{x}(t) y(t)) dt} = 4A \geq 0,$$
where $A$ is the area of the domain enclosed by the curve $\gamma$ (weighted at each point with the winding number with respect to $\gamma$).
\end{proof}
If $F = B \cdot G$ has more than one root in $\mathbb{D}$, Lemma 1 can be applied iteratively.

\begin{proof}[Proof of Theorem 2.] The previous language establishes a relationship between $f,g$ and $F$. However, only the \textit{modulus} of $F$ ever appears in the argument: exploiting that 
$$ |z-\alpha_i| = |1-\overline{\alpha_i}z| \qquad \mbox{whenever} \quad |z| = 1$$
allows for a better using of the gain obtained from iterative application of the previous Lemma when inverting several roots along the unit circle. More precisely, consider again
\begin{align*}
F(z) &= (z-\alpha_1)(z-\alpha_2)(z-\alpha_3)\dots (z-\alpha_{n-1})  (z-\alpha_n) H(z) \\
F_1(z) &= (1-\overline{\alpha_1}z)(z-\alpha_2)(z-\alpha_3)\dots  (z-\alpha_{n-1})  (z-\alpha_n) H(z) \\
F_2(z) &= (1-\overline{\alpha_1}z)(1-\overline{\alpha_2}z)(1-\overline{\alpha_3}z)\dots  (z-\alpha_{n-1})  (z-\alpha_n) H(z) \\
\dots \\
F_{n-1}(z)&= (1-\overline{\alpha_1}z)(1-\overline{\alpha_2}z)(1-\overline{\alpha_3}z)\dots  (1-\overline{\alpha_{n-1}}z)  (z-\alpha_n) H(z) 
\end{align*}
The crucial new ingredient is that 
$$ \|F(z)\big|_{\partial \mathbb{D}}\|_{L^2(\partial \mathbb{D})} =  \|F_1(z)\big|_{\partial \mathbb{D}}\|_{L^2(\partial \mathbb{D})} = \dots =  \|F_{n-1}(z)\big|_{\partial \mathbb{D}}\|_{L^2(\partial \mathbb{D})}.$$
The very same reason allows for a more precise analysis of the effect removing one root has.  Let again $|\alpha| < 1$ be the root; we compare
$$ f(z) = (z-\alpha)F(z) \qquad \mbox{and} \qquad g(z) = (1-\overline{\alpha}z)F(z)$$
on the boundary $\partial \mathbb{D}$.
The same computation as before yields
$$ \|f(z)\big|_{\partial \mathbb{D}}\|^2_{\mathcal{D}} - \|g(z)\big|_{\partial \mathbb{D}}\|^2_{\mathcal{D}} = (1-|\alpha|^2)\|F\big|_{\partial \mathbb{D}}\|^2_{L^2} =  (1-|\alpha|^2)\left\|\frac{f}{z-\alpha}\big|_{\partial \mathbb{D}}\right\|^2_{L^2} .
$$
In particular, all the arising expressions can be summed in closed form and the arising gain is 
$$\|F\|_{\mathcal{D}}^2 - \|G\|_{\mathcal{D}}^2 = \int_{0}^{2\pi}{|G(e^{it})|^2 \sum_{j}{\frac{1-|a_j|^2}{|e^{it}-\alpha_j|^2}}dt}.$$
\end{proof}
There is a difference of a factor 2 in the way we stated the Carleson's formula and the proof above: this is due to the fact that we computed the effect on what
turns out to be $H^{1/2}$ whereas the Dirichlet space in Carleson's formula also has the $\mathcal{H}^2-$norm (which stays preserved since $|F| = |G|$),
hence the difference of a factor of 2.

\subsection{Proof of Corollary \ref{fullderiv}} We have 
$$ \int_{0}^{2\pi}{| G_{n+1}'(e^{i\theta})|^2 d\theta} \leq  \int_{0}^{2\pi}{| G_n'(e^{i\theta})|^2 d\theta}  - \int_{\partial \mathbb{D}}{|G_n(e^{i\theta})-G_n(0)|^2 d\theta},$$
but there is no way of turning this into a quantitative decay estimate because the gain
$$  \int_{\partial \mathbb{D}}{|G_n(e^{i\theta})-G_n(0)|^2 d\theta} \qquad \mbox{need not be proportional to the size of} \qquad  \int_{0}^{2\pi}{| G_n'(e^{i\theta})|^2 d\theta} .$$
Put geometrically, $G_n(e^{i\theta})$ may wind around $G_n(0)$ very quickly while $|G_n(e^{i\theta})-G_n(0)|$ could be quite small all the time. The crucial insight is as follows: if that were actually the case and $\|G_n(e^{i\theta})-G_n(0)\|_{L^2(\mathbb{T})}$ is small,
then one would certainly hope that $\|G_n(e^{i\theta})-G_n(0)\|_{L^{\infty}(\mathbb{T})}$ is also small. 
Now we reverse the order of that argument: suppose that $\|G_n(e^{i\theta})-G_n(0)\|_{L^{\infty}(\mathbb{T})}$ is \textit{not} small. This means that $|G_n(e^{i\theta})-G_n(0)|$ is
big for some $\theta$, which does not at all mean that the function is large in $L^2$, it could just be big in one place and very small everywhere else: this, however, would imply
that the $L^2-$norm of the gradient is large and we know it cannot exceed that of the initial data.

\subsubsection{Sobolev embedding.} The second ingredient of the argument may be formulated as follows: let $h:\mathbb{T} \rightarrow \mathbb{R}$ be differentiable. If $\|h'\|_{L^2(\mathbb{T})}$ is not very big and $\|h\|_{L^{\infty}(\mathbb{T})}$ has a certain size,
then $\|h\|_{L^2{(\mathbb{T})}}$ cannot be arbitrarily small (depending on the first two quantities): the only way to be big in $L^{\infty}(\mathbb{T})$ but small in $L^2(\mathbb{T})$ is quick decay around the point
where the supremum is assumed. The inequality is merely the classical embedding of the Sobolev space $H^1(\mathbb{T}) \hookrightarrow L^{\infty}(\mathbb{T})$.
\begin{lemma} Let $h:\mathbb{T} \rightarrow \mathbb{R}$ be a differentiable function which changes sign. Then
$$ \|h\|_{L^2}^2 \geq \frac{1}{16} \frac{\|h\|^4_{L^{\infty}}}{\|h'\|^2_{L^2}}.$$
\end{lemma}
\begin{proof} Assume without loss of generality\ that $h(0) = 0$. Assume $x$ to be such that $|h(x)| = \|h\|_{L^{\infty}}$. 
Using the Cauchy-Schwarz inequality, we get
$$ \|h\|_{L^{\infty}(\mathbb{T})}^{2} = |h(x)|^2  \leq 4  \left(\int_{0}^{x}{ |h(z)|^2 dz} \right)^{\frac{1}{2}} \left(\int_{0}^{x}{ |h'(z)|^2 dz} \right)^{\frac{1}{2}}.$$
Squaring both sides gives the result.
\end{proof}
\subsubsection{Random projections.} 

We apply the statement to a curve $\gamma_F:\partial \mathbb{D} \rightarrow \mathbb{C}$, which is different object than a periodic function $h:\mathbb{T} \rightarrow \mathbb{R}$. The natural 
approach to reduce one to the other would be to fix a vector $n \in \mathbb{R}^2$ with unit length $|\nu| = 1$ and consider the projection
$$ h(t) = \left\langle \gamma(t), \nu \right\rangle.$$
The next Lemma states that there exists a unit vector $\nu$ such this reduction does not change the $L^2-$norm and $L^{\infty}-$norm by more than an absolute constant:
$$\int_{0}^{2\pi}{|\gamma(t)|^2 dt} \leq  6  \int_{0}^{2\pi}{|\left\langle \nu,\gamma(t)\right\rangle|^2 dt}. \quad \mbox{and} \quad |h(0)|= |\left\langle \gamma(0), \nu\right\rangle| \geq \frac{1}{\sqrt{2}} |\gamma(0)|.$$

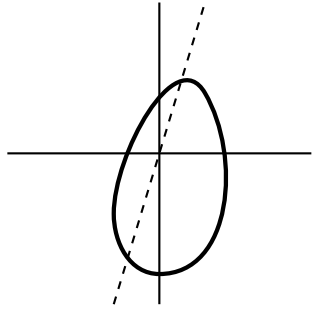
\begin{figure}[h!]
\begin{center}
\begin{tikzpicture}[scale=2]
\draw [thick] (-1,0) -- (1,0);
\draw [thick] (0,-1) -- (0,1);
\draw [ultra thick] (0,-0.8) to [out=0,in=300] (0.3,0.4)  to [out=120,in=90] (-0.3,-0.4)  to [out=270,in=180] (0,-0.8);
\draw [thick, dashed] (-0.3,-1) -- (0.3,1);
\end{tikzpicture}
\end{center}
\caption{Assuming arclength-parametrization, projecting the curve onto the dashed line preserves both $L^2-$norm and the $L^{\infty}-$norm up to absolute constants.}
\end{figure}

It is easy to see by Cauchy-Schwarz that $h$ varies slower than $\gamma$
$$ \left|\frac{d}{dt}h(t)\right| = \left|\left\langle \gamma'(t), \nu\right\rangle \right| \leq |\gamma'(t)|.$$
Therefore, after having established the existence of such a vector $n$ and reparametrizing the curve in such a way that $|h(0)| \geq 2^{-1/2}\max_t |\gamma(t)|$, we could deduce that
$$  \int_{0}^{2\pi}{|\gamma(t)|^2 dt} \geq \int_{0}^{2\pi}{|h(t)|^2 dt}  \geq \frac{1}{16}\frac{\|h\|^4_{L^{\infty}}}{\|h'\|^2_{L^2}} \geq \frac{1}{64} \frac{\max_t |\gamma(t)|^4}{\int_{0}^{2\pi}{|\gamma'(z)|^2 dz}}$$
which is a quantitative version of our intuition described above: in order for the function to be big at some point, it cannot be too small on average. Let us now prove the statement. The argument says that it is sufficient to take that vector at random to have the desired property to be true on average
(in particular, there exists at least one vector for which it is true).
\begin{lemma} Let $\gamma:\mathbb{T} \rightarrow \mathbb{R}^2$ be a periodic curve in the plane and assume $\gamma(0) \neq (0,0)$. Then there exists a unit vector $|\nu| = 1$ with
$$ |\left\langle \gamma(0), \nu\right\rangle| \geq \frac{1}{\sqrt{2}} |\gamma(0)|$$
as well as
$$ \int_{0}^{2\pi}{|\gamma(t)|^2 dt} \leq  6  \int_{0}^{2\pi}{|\left\langle \nu,\gamma(t)\right\rangle|^2 dt}. $$
\end{lemma}
\begin{proof} The line from the origin to $\gamma(0)$ defines a unique angle $\alpha$. Let us now chose $\nu$ randomly from $\alpha - \pi/4$ and $\alpha + \pi/4$. Any such vector
satisfies the first condition and we will now compute the expectation of the $L^2-$norm for such a random vector. We first remark that for every fixed vector $\ell \in \mathbb{R}^2$ and every
$0 \leq \alpha \leq 2\pi$ a simple computation shows that
\begin{align*}
\int_{\alpha-\frac{\pi}{4}}^{\alpha+\frac{\pi}{4}}{|\left\langle (\cos{s}, \sin{s}),\ell\right\rangle|^2 ds} &\geq |\ell|^2 \int_{-\frac{\pi}{4}}^{\frac{\pi}{4}}{|\left\langle (\cos{s}, \sin{s}),(0,1)\right\rangle|^2 ds}\\
&\geq  |\ell|^2  \int_{-\frac{\pi}{4}}^{\frac{\pi}{4}}{\left(\sin{s}\right)^2ds} \\
&= |\ell|^2 \frac{\pi-2}{4}.
\end{align*}
We now compute the expectation by using this and exchanging the order of integration 
\begin{align*}
\frac{2}{\pi}\int_{\alpha-\frac{\pi}{4}}^{\alpha+\frac{\pi}{4}}{\int_{0}^{2\pi}{|\left\langle (\cos{s}, \sin{s}),\gamma(t)\right\rangle|^2 dt}ds} &=
\frac{2}{\pi}  \int_{0}^{2\pi}{   \int_{\alpha-\frac{\pi}{4}}^{\alpha+\frac{\pi}{4}}{  |\left\langle (\cos{s}, \sin{s}),\gamma(t)\right\rangle|^2 ds}dt} \\
 &\geq \frac{2}{\pi} \frac{\pi -2}{4} \int_{0}^{2\pi}{|\gamma(t)|^2dt}.
\end{align*}
Since $(2\pi)/(\pi-2) \leq 6$ and since a random vector has that expectation, there exists at least one vector with that value.
\end{proof}

\subsubsection{Proof of Corollary \ref{fullderiv}.} 
The monotonicity formula implies that
$$  \int_{0}^{2\pi}{|(G_n(e^{it})-G_n(0))'|^2 dt } = \int_{0}^{2\pi}{|G_n'(e^{it})|^2 dt} \qquad \mbox{is monotonically decreasing in}~n.$$
Suppose that for some $n$ and some $z \in \partial \mathbb{D}$
$$ |G_n(z) - G_n(0)| \geq \varepsilon.$$
We can identify $G_n(z)-G_n(0):\partial \mathbb{D} \rightarrow \mathbb{C}$ with a curve $\gamma:\mathbb{T} \rightarrow \mathbb{R}^2$ and reparametrize it using our Lemma so that
$$ |\left\langle \gamma(0), \nu\right\rangle | \geq \frac{\varepsilon}{\sqrt{2}}$$
and
$$ \int_{0}^{2\pi}{|\gamma(t)|^2 dt} \leq  6\int_{0}^{2\pi}{|\left\langle \nu,\gamma(t)\right\rangle|^2 dt}. $$
We can now apply our second Lemma to the function
$$ h(t) = \left\langle \gamma(t), \nu\right\rangle.$$
Since $G_n(z)-G_n(0)$ has winding number at least 1, so has $\gamma$ and therefore $h$ vanishes at least in two points. $h$ and $G_n(z) - G_n(0)$ have comparable $L^{\infty}(\mathbb{T})-$norm (up to a factor of $\sqrt{2}$) and comparable $L^2(\mathbb{T})-$norm (up to a factor of 6)
and elementary geometric considerations (projections make vectors only smaller) show that the derivative of $h$ satisfies
$$ |h'(t)| \leq | \gamma'(t)|.$$
We can now apply Lemma 2 and conclude that
$$  \int_{0}^{2\pi}{|(G_n(e^{it})-G_n(0))|^2 dt} \sim \|h\|_{L^2}^2 \gtrsim \frac{\|h\|^4_{L^{\infty}}}{\|h'\|^2_{L^2}} \geq \frac{\varepsilon^4}{ \int_{0}^{2\pi}{|F'(e^{i\theta})|^2 d\theta} },$$
where the second inequality follows from the assumption that $ |G_n(z) - G_n(0)| \geq \varepsilon$ and the fact that the $L^2-$norm of the derivative is decreasing.
Now, let's look at the next step in the algorithm, where we decompose
$$ G_n - G_n(0) = B G_{n+1}.$$
Our inequality tells us that the squared $L^2-$norm of the derivative decreases at least by a factor of (using $|B|=1$ for Blaschke products)
$$ \int_{0}^{2\pi}{|G_{n+1}(e^{it})|^2 dt} =  \int_{0}^{2\pi}{|G_n(e^{it}) - G_n(0) |^2 dt} \gtrsim \frac{\varepsilon^4}{ \int_{0}^{2\pi}{|F'(e^{i\theta})|^2 d\theta} }.$$
This yields
\begin{align*}
 \int_{0}^{2\pi}{| G_{n+1}'(e^{i\theta})|^2 dt} &\leq  \int_{0}^{2\pi}{| G_n'(e^{i\theta})|^2 dt}  - \int_{0}^{2\pi}{|G_n(e^{it})-G_n(0)|^2 dt}  \\
&\leq  \int_{0}^{2\pi}{| G_n'(e^{i\theta})|^2 dt}  - c \frac{\varepsilon^4}{ \int_{0}^{2\pi}{|F'(e^{i\theta})|^2 d\theta} }
\end{align*}
for some universal constant $c  > 0$.
However, since all the involved quantities are nonnegative, this immediately implies that number of $n \in \mathbb{N}$ for which 
$$ |G_n(z) - G_n(0)| \geq \varepsilon \qquad \mbox{for some}~z \in \partial \mathbb{D}$$
is bounded from above by
$$ \left( \int_{0}^{2\pi}{|F'(e^{i\theta})|^2 d\theta} \right) / \left( c \frac{\varepsilon^4}{ \int_{0}^{2\pi}{|F'(e^{i\theta})|^2 d\theta} } \right) =\frac{1}{c}\frac{\left( \int_{0}^{2\pi}{|F'(e^{i\theta})|^2 d\theta}\right)^2 }{\varepsilon^4}.$$
This concludes the argument. $\qed$

\subsection{Proof of Corollary \ref{halfformula}} 
 We recall the action of removing one root which entails comparing
$$ f(z) = (z-\alpha)F(z) \qquad \mbox{and} \qquad g(z) = (1-\overline{\alpha}z)F(z).$$
As was shown in the proof of Theorem \ref{main}, we have the identity
$$ \|f(z)\big|_{\partial \mathbb{D}}\|^2_{X} - \|g(z)\big|_{\partial \mathbb{D}}\|^2_{X} = (1-|\alpha|^2)\|F\big|_{\partial \mathbb{D}}\|^2_{Y}. $$
In the Dirichlet space $X = \mathcal{D}$, we have $\gamma_n = n$ and thus $Y = L^2$. In the proof of Theorem 1, we used the monotonicity formula to remove all roots
and applied the full strength of the inequality only once for a root that is in the origin. Here, the special algebraic structure of the space allows us to apply the inequality
multiple times \textit{and sum all the contributions in closed form}. 
The crucial ingredient that makes this possible is the algebraic identity on $\partial \mathbb{D}$ for all $|\alpha|, |\beta| < 1$
$$ |z-\alpha||z-\beta| = |1-\overline{\alpha}z||z-\beta|  = |1-\overline{\alpha}z||1-\overline{\beta}z| \qquad \mbox{whenever}~|z|=1.$$
We will now illustrate the effect of applying the
identity twice (to remove two roots $\alpha, \beta$ from $\mathbb{D}$).
 The arising functions are
$$ f_1(z) = (z-\alpha) (z-\beta)F(z), \quad f_2(z) = (1-\overline{\alpha}z) (z-\beta)F(z)\quad \mbox{and} \quad  f_3(z) = (1-\overline{\alpha}z)  (1-\overline{\beta}z) F(z).$$
Applying the identity twice yields
\begin{align*}
 \|f_1(z)\big|_{\partial \mathbb{D}}\|^2_{\mathcal{D}} - \|f_2(z)\big|_{\partial \mathbb{D}}\|^2_{\mathcal{D}} &= \|(z-\beta)F\big|_{\partial \mathbb{D}}\|^2_{L^2}  \\
 \|f_2(z)\big|_{\partial \mathbb{D}}\|^2_{\mathcal{D}} - \|f_3(z)\big|_{\partial \mathbb{D}}\|^2_{\mathcal{D}} &= \|(1-\overline{\alpha}z)F\big|_{\partial \mathbb{D}}\|^2_{L^2}
\end{align*}
Normally, we would be unable to sum up these two contributions, however, here the algebraic identity implies that
$$  \|(z-\beta)F\big|_{\partial \mathbb{D}}\|^2_{L^2} =  \left\| \frac{f_1(z)}{z-\alpha}\big|_{\partial \mathbb{D}}\right\|^2_{L^2}$$
and
$$ \|(1-\overline{\alpha}z)F\big|_{\partial \mathbb{D}}\|^2_{L^2} = \|(z-\alpha)F\big|_{\partial \mathbb{D}}\|^2_{L^2} =  \left\| \frac{f_1(z)}{z-\beta}\big|_{\partial \mathbb{D}}\right\|^2_{L^2}$$
This allows us to simplify
$$  \|(z-\beta)F\big|_{\partial \mathbb{D}}\|^2_{L^2} +  \|(1-\overline{\alpha}z)F\big|_{\partial \mathbb{D}}\|^2_{L^2} = \int_{\partial \mathbb{D}}{|f_3(z)|^2 \left( \frac{1}{|z-\alpha|^2} +  \frac{1}{|z-\beta|^2} \right)}.$$
 The general case for more sums follows by the same reasoning. $\qed$\\

The argument can be easily summarized as saying that the algebraic structure of $X = \mathcal{D}$ implies that $Y = L^2$; the additional
algebraic ingredient is $|B(e^{i\theta})| = 1$ which implies that the various $\|F\big|_{\partial \mathbb{D}}\|^2_{Y}$ one gets from successive
removal of roots can actually be summed up in closed form.

\subsection{Proof of Theorem \ref{rootstab}}
\begin{proof} The statement is pointwise and invariant under multiplication with polynomials having all roots outside of $\mathbb{D}$: it thus suffices to prove it for polynomials having all their roots inside of $\mathbb{D}$. We write
$$ F_1 = \prod_{i=1}^n{(z-\alpha_i)}   \qquad \mbox{and} \qquad F_2 = \prod_{i=1}^n{(z-\beta_i)}.$$
Obviously
$$ B_1 = \prod_{i=1}^n{\frac{z-\alpha_i}{1-\overline{\alpha_i}z}}   \qquad \mbox{as well as} \qquad B_2 =  \prod_{i=1}^n{\frac{z-\beta_i}{1-\overline{\beta_i}z}}$$
and thus
$$ G_1 = \prod_{i=1}^n{(1-\overline{\alpha_i}z)}   \qquad \mbox{as well as} \qquad G_2 =  \prod_{i=1}^n{(1-\overline{\beta_i}z)}.$$
An explicit computation yields that
$$ G_1(z) - G_2(z) = \sum_{k=0}^{n}{z^k \left( \sum_{A \subset \left\{1, \dots, n\right\} \atop |A| = k}{\prod_{j \in A}{(-\overline{\alpha_j})}} 
-   \sum_{B \subset \left\{1, \dots, n\right\} \atop |B| = k}{\prod_{j \in B}{(-\overline{\beta_j})}} \right)}$$
while
$$ F_1(z) - F_2(z) = \sum_{k=0}^{n}{z^k \left( \sum_{A \subset \left\{1, \dots, n\right\} \atop |A| = n-k}{\prod_{j \in A}{(-\alpha_j)}} 
-   \sum_{B \subset \left\{1, \dots, n\right\} \atop |B| = n-k}{\prod_{j \in B}{(-\beta_j)}} \right)}.$$
Altogether, this implies that if
$$ G_1(z) - G_2(z) = \sum_{k=1}^{n}{c_k z^k} \qquad \mbox{then} \qquad F_1(z) - F_2(z) = \sum_{k=1}^{n}{\overline{c_{n-k}} z^k}.$$
It remains to show that both quantities have the same norm if $|z| = 1$
\begin{align*}
|F_1(z) - F_2(z)|  &=  \left|  \sum_{k=1}^{n}{\overline{c_{n-k}} z^k} \right| =   \left| \frac{1}{z^n} \sum_{k=1}^{n}{\overline{c_{n-k}} z^k} \right| =  \left|  \sum_{k=1}^{n}{\overline{c_{n-k}} \left(\frac{1}{z}\right)^{n-k}} \right| \\ 
&=  \left|  \sum_{k=1}^{n}{\overline{c_{n-k}} ~\overline{z}^{n-k}} \right| =  \left|  \sum_{k=1}^{n}{c_{n-k}z^{n-k}} \right| = \left| G_1(z) -G_2(z) \right|.
\end{align*}
\end{proof}

\subsection{Proof of Theorem \ref{complex}}
\begin{proof} We imitate the argument in the case of Fourier series and again study the effect of removing one root by comparing
$$ f(z) = (z-\alpha)F(z) \qquad \mbox{and} \qquad g(z) = (z-\overline{\alpha})F(z).$$
Note that
$$ \widehat{F}(\xi) = i \frac{d}{d\xi} \widehat{F}(\xi) - \alpha \widehat{F}(\xi)  \qquad \mbox{and} \qquad \widehat{g}(\xi) = i \frac{d}{d\xi} \widehat{F}(\xi) - \overline{\alpha} \widehat{F}(\xi).$$
and therefore
\begin{align*}
\|f\|_{X}^2 - \|g\|_{X}^2 &= \int_{0}^{\infty}{\left(\left| i \frac{d}{d\xi} \widehat{F}(\xi) - \alpha \widehat{F}(\xi)\right|^2 - \left|  i \frac{d}{d\xi} \widehat{F}(\xi) - \overline{\alpha} \widehat{F}(\xi)\right|^2\right) \psi(\xi) d\xi}
\end{align*}
After simple computation we arrive at
$$\left| i \frac{d}{d\xi} \widehat{F}(\xi) - \alpha \widehat{F}(\xi)\right|^2 - \left|  i \frac{d}{d\xi} \widehat{F}(\xi) - \overline{\alpha} \widehat{F}(\xi)\right|^2 = i(\alpha - \overline{\alpha})\left[\widehat{F}' \overline{\widehat{F}} + \widehat{F} \overline{\widehat{F}'}\right].$$
Writing $\widehat{F}(\xi) = a(\xi) + i b(\xi)$ as real and imaginary parts, we see that
$$\widehat{F}' \overline{\widehat{F}} + \widehat{F} \overline{\widehat{F}'}=  2( a'a + b'b).$$
This implies that we can write
$$\left| i \frac{d}{d\xi} \widehat{F}(\xi) - \alpha \widehat{F}(\xi)\right|^2 - \left|  i \frac{d}{d\xi} \widehat{F}(\xi) - \overline{\alpha} \widehat{F}(\xi) \right|^2 =i(\alpha - \overline{\alpha}) \frac{d}{d \xi} |\widehat F(\xi)|^2$$
and therefore with integration by parts
\begin{align*}
\|f\|_{X}^2 - \|g\|_{X}^2 = \int_{0}^{\infty}{\left(i(\alpha - \overline{\alpha}) \frac{d}{d \xi} |\widehat F(\xi)|^2\right) \psi(\xi) d\xi} = 2 \Im \alpha  \int_{0}^{\infty}{ |\widehat F(\xi)|^2 \psi'(\xi) d\xi},
\end{align*}
which is clearly nonnegative because $\alpha \in \mathbb{C}_{+}$. It remains to study the special case of the Dirichlet space: if $\psi(\xi) = \xi$ we have with Plancherel that
$$ \int_{0}^{\infty}{ |\widehat F(\xi)|^2 \psi'(\xi) d\xi} = \int_{0}^{\infty}{ |\widehat{F}(\xi)|^2 d\xi} = \int_{\mathbb{R}}^{}{|F(\xi)|^2 d\xi}.$$
The key ingredient is again of an algebraic nature: the difference can be quantified in terms of a quantity whose behavior can be controlled while removing several roots one after
the other. Let us illustrate this again with
$$ f_1(z) = (z-\alpha) (z-\beta)F(z), \quad f_2(z) = (z-\overline{\alpha}) (z-\beta)F(z)\quad \mbox{and} \quad  f_3(z) = (z-\overline{\alpha})(z-\overline{\beta}) F(z).$$
Applying the identity twice yields
\begin{align*}
 \|f_1(z)\|^2_{\mathcal{D}} - \|f_2(z)\|^2_{\mathcal{D}} &= \|(z-\beta)F\|^2_{L^2}  \\
 \|f_2(z)\|^2_{\mathcal{D}} - \|f_3(z)\|^2_{\mathcal{D}} &= \|(z-\overline{\alpha})F\|^2_{L^2}
\end{align*}
and we see once more that the sum of the gain can be controlled. This yields the statement.
\end{proof}

\end{document}